\documentclass[reqno]{amsart}
\usepackage{tabu}
\usepackage{amssymb}
\usepackage{mathtools}
\usepackage{a4wide,amsmath}
\usepackage{mathrsfs}
\usepackage{amsthm}
\numberwithin{equation}{section}
\numberwithin{figure}{section}
\numberwithin{table}{section}
\usepackage{bbm}
\usepackage{subfig}
\usepackage{needspace}
\usepackage{hyphenat}
\usepackage{comment}
\usepackage{enumerate}
\usepackage[section]{placeins}
\usepackage{graphicx}		  % Tilføjelse af billedfiler
\usepackage{ifpdf}
\ifpdf
\DeclareGraphicsExtensions{.pdf,.eps,.jpg,.png}	
\usepackage[suffix=]{epstopdf}
\fi
\usepackage{xcolor}
\usepackage[utf8]{inputenc}
\usepackage{hyperref}
\hypersetup{hidelinks}

\long\def\MSC#1\EndMSC{\def\arg{#1}\ifx\arg\empty\relax\else
	{\narrower\noindent%
		{2020 Mathematics Subject Classification}: #1\\} \fi}
\long\def\PACS#1\EndPACS{\def\arg{#1}\ifx\arg\empty\relax\else
	{\narrower\noindent%
		{PACS numbers}: #1}\fi}
\long\def\KEY#1\EndKEY{\def\arg{#1}\ifx\arg\empty\relax\else
	{\narrower\noindent% 
		Keywords: #1\\}\fi}

\newcommand{\norm}[1]{\lVert#1\rVert}  
\newcommand{\abs}[1]{\lvert#1\rvert} 
\newcommand{\inner}[1]{\langle#1\rangle} 
\newcommand{\di}{\mathrm{d}}  
\newcommand{\DLambda}{\textup{D}\mkern-1.5mu \Lambda}
\newcommand{\DhLambda}{\textup{D}\mkern-1.5mu \widehat{\Lambda}}
\newcommand{\D}{\mathcal{D}}

\newcommand{\R}{\mathbb{R}}
\newcommand{\N}{\mathbb{N}}
\newcommand{\C}{\mathbf{C}}

\newcommand{\p}{\partial}

\newcommand\Tr{\operatorname{tr}} % Trace
\newcommand\supp{\operatorname{supp}} 
\newcommand\Span{\operatorname{span}}

\theoremstyle{definition}

\theoremstyle{plain}
\newtheorem{thm}{Theorem}[section]
\newtheorem{lem}[thm]{Lemma}
\newtheorem{prop}[thm]{Proposition}

\theoremstyle{definition}
\newtheorem{defn}[thm]{Definition}
\newtheorem{assumption}[thm]{Assumption}
\theoremstyle{remark}
\newtheorem{rem}[thm]{Remark}

\numberwithin{equation}{section}

\begin{document}

\title[Reconstruction of shear modulus inclusions]{Reconstruction of shear modulus inclusions in elastostatics via divergence-free localization}
\date{}

\author[H.~Garde]{Henrik Garde}
\address[H.~Garde]{Department of Mathematics, Aarhus University, Aarhus, Denmark.}
\email{garde@math.au.dk}

\author[N.~Hyv\"onen]{Nuutti Hyv\"onen}
\address[N.~Hyv\"onen]{Department of Mathematics and Systems Analysis, Aalto University, Helsinki, Finland.}
\email{nuutti.hyvonen@aalto.fi}

\author[V.~Pohjola]{Valter Pohjola}
\address[V.~Pohjola]{Applied and Computational Mathematics, University of Oulu, Oulu, Finland}
\email{valter.pohjola@gmail.com}

\begin{abstract}
	We formulate an improved version of the monotonicity method for shape reconstruction in the inverse problem of linear elastostatics. We show that this method can reconstruct inclusions in the second Lam\'e parameter,~i.e.,~the shear modulus, independently of inhomogeneities in the first Lam\'e parameter. This improves on earlier methods that only recover inclusions in the bulk modulus under certain assumptions on the local interplay between the Lam\'e parameters.  Our proofs are based on linearization and divergence-free localization of energy using harmonic vector fields.
\end{abstract}

\maketitle
\tableofcontents

\section{Introduction} 

We consider the inverse boundary value problem in linear elastostatics: determine the Lam\'e parameters $\lambda$ and $\mu$ of an elastic body $\Omega \subset \R^3$, from boundary measurements of tractions and the corresponding displacements; see,~e.g.,~\cite{Barbone04, Beretta14a, Beretta14b, Carstea18, Eskin02, Ikehata90, Ikehata06, Ikehata99, Imanuvilov11,Lin17,Nakamura99,Nakamura93,Nakamura94,Nakamura94erratum,Nakamura95} for previous theoretical considerations on this problem. In this work, we focus on inclusion detection,~i.e.,~reconstructing the regions where $\lambda$ and $\mu$ deviate from prescribed background values, by means of the \emph{monotonicity method}~\cite{Tamburrino02,Harrach2008,HU13,Garde2020}. Our aim is to detect inclusions in the second Lam\'e parameter,~i.e.,~the shear modulus~$\mu$, independently of the first Lam\'e parameter $\lambda$. 

The previous works on the monotonicity method in the framework of linear elastostatics have only considered the detection of inclusions where the Lam\'e parameters satisfy $(\mu-\mu_0)(\lambda-\lambda_0) \geq 0$ in comparison to known  background parameters $\lambda_0$ and $\mu_0$ (see, in particular, \cite{EGH25,EH21}). This corresponds to finding inclusions in the bulk modulus $\lambda + \tfrac{2}{3}\mu$ under further restrictions on the interplay of the Lam\'e parameters. Combined with the earlier results, our work allows a partial separation of inclusions in the shear and bulk moduli; see Figure~\ref{fig_intro}. The shear and bulk moduli have natural physical interpretations: the former measures a material's resistance to shape change when subjected to a force applied parallel to its surface and the latter measures the resistance to uniform compression (see,~e.g.,~\cite{OSY92}). 

\begin{figure}[htb]
	\centering
	\includegraphics[width=0.9\textwidth]{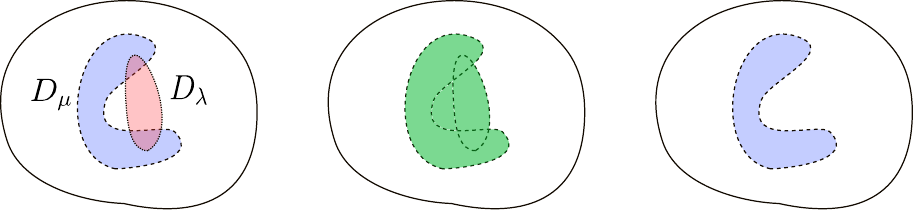}
	\caption{Left: The true structure of the object, with perturbations in the Lamé parameters in the regions $D_\lambda$ and $D_\mu$. For simplicity, assume $\lambda -\lambda_0>0$ in $D_\lambda$  and $\mu-\mu_0>0$ in $D_\mu$, which are sufficient, but not necessary, conditions for the functionality of the monotonicity method. Middle: Previous methods give a reconstruction of the outer shape of the inclusions in the bulk modulus. Right: The improved method presented here separates the inclusions in the shear modulus.}\label{fig_intro}
\end{figure}

The monotonicity method was introduced to the framework of linear elastostatics in~\cite{EH21}, where the so-called \emph{inner approach} was used to determine if a given open set is fully contained in an inclusion that is characterized by either increase or decrease in {\em both} Lam\'e parameters. We call an inclusion with and an increase (decrease) in a Lamé parameter from a background value,  a positive (negative) inclusion. In particular, the results on inclusion detection in \cite{EH21} only allow inclusions with a same sign inside $\Omega$. By resorting to the \emph{outer approach} of the monotonicity method, \cite{EGH25} provided a characterization of inclusions composed of both positive and negative parts, allowing even perfectly elastic or infinitely stiff inclusions. However, \cite{EGH25} also
only consideres inclusions that are locally positive or negative, that is, the perturbations in $\mu$ and $\lambda$ are required to have the same sign locally, meaning that the expression $(\mu-\mu_0)(\lambda-\lambda_0)$ is assumed to be non-negative throughout the domain. 

In this work, we decouple $\mu$ from $\lambda$ so that we reconstruct inclusions in $\mu$ without possible perturbations in  $\lambda$ having any effect on this process. Indeed, $\lambda$ can be completely arbitrary, although we do require the knowledge of lower and upper bounds on $\lambda$. The outer approach (Theorem~\ref{thm:outer}) enables reconstructing inclusions in $\mu$ for which the perturbations $\mu-\mu_0$ can have both positive and negative parts. We also consider the more specialized inner approach (Propositions~\ref{prop:posinc} and~\ref{prop:neginc}) that characterizes positive (or negative) inclusions in $\mu$. 

Choosing suitable function spaces is crucial for our analysis, and we thus formulate the forward problem of linear elastostatics in a quotient space $H^1(\Omega)^3/\mathcal{R}$ (see, e.g.,~\cite{Ciarlet78, Kaleem26}), where $\mathcal{R}$ denotes the finite-dimensional space of rigid motions, or more precisely in the space $\mathcal{H}(\Omega)$  (Definition~\ref{eq_def_H}) of the representatives for the
equivalence classes in $H^1(\Omega)^3/\mathcal{R}$ with centered displacements~\cite[section 5]{Kaleem26}. This functional-analytic setting replaces the auxiliary Dirichlet boundary conditions often imposed in elastostatics, and it is particularly well suited for our localized fields. Enforcing additional Dirichlet conditions would be incompatible with the divergence-free constructions that drive our monotonicity arguments; see Remarks~\ref{rem_bc} and  \ref{rem_dirichlet}.

While the reconstruction in $\mu$, separately from $\lambda$, gives a generalization of previous results in terms of not requiring $(\mu-\mu_0)(\lambda-\lambda_0)$ to be non-negative, but with $\lambda$ allowed to be arbitrary, the main downside is that we require the background shear modulus $\mu_0$ to be a constant function. This enables the construction of sufficiently rich families of localized divergence-free solutions (Theorem~\ref{thm_loc}), which are the key to disentangling the effects of $\mu$ from those by $\lambda$ in the linearized monotonicity inequalities. An analogous result was proven in \cite{EP26} for the case of time-harmonic elastic waves. The time-harmonic case differs significantly from its elastostatic counterpart and the ideas from \cite{EP26} cannot be directly imported to our setting. For more information on the time harmonic setting, see~\cite{EP26,EP23,HPS19b,HPS19a}.

Our localized probing fields for detecting perturbations only in $\mu$ are harmonic vector fields, that is, solutions $u$ to the forward problem of linear elastostatics with $\nabla \times u = 0$ and $\nabla \cdot u = 0$, realized as gradients $u = \nabla \varphi$ of harmonic potentials $\varphi$. These potentials are required to satisfy certain boundary conditions, ensuring that $u$ belongs to $\mathcal{H}(\Omega)$. To achieve this, we devise a constrained Runge approximation scheme tailored to elastostatics in $\mathcal{H}(\Omega)$, enforcing the required compatibility with rigid motions. In the present setting, the constrained approximation is employed not to avoid resonance phenomena as in~\cite{EP23,EP26,HPS19b}, but to produce localized, divergence-free test fields that respect the quotient-space formulation. They can be used directly in the linearized monotonicity inequalities to decouple perturbations in $\mu$ from those in $\lambda$.

The article is organized as follows. Section~\ref{sec:preliminaries} considers the employed notation and other preliminaries needed for introducing our mathematical setting. The pure traction forward problem of linear elastostatics is described in Section~\ref{sec_direct}. Our main results are given in Section~\ref{sec:main}, which includes the localization result for divergence-free solutions in Theorem~\ref{thm_loc}, and the reconstruction methods for the outer approach in Theorem~\ref{thm:outer} and the inner approach in Propositions~\ref{prop:posinc} and~\ref{prop:neginc}. Section~\ref{sec:divergence_free} is devoted to the proof of the localization result, while the proofs of the reconstruction methods are given in Sections~\ref{sec:proofthmouter}--\ref{sec:proofneginc}.  

\section{Preliminaries}
\label{sec:preliminaries}

Recall that the Frobenius inner product of matrices is defined as
\begin{equation} \label{eq:inner}
	A:B = \sum_{i,j} A_{ij}B_{ij}, \quad A,B \in \R^{m\times n},
\end{equation}
and denote the induced norm by
\begin{equation} \label{eq:norm}
	\abs{A} = (A:A)^{1/2}, \quad A \in \R^{m \times n}.
\end{equation}
The definitions \eqref{eq:inner}--\eqref{eq:norm} are consistent with the Euclidean inner product and norm for vectors as well as the absolute value for scalars. For vectors we will also use the usual dot product notation for the inner product.

A rigid motion in $\Omega$ is defined as
\begin{equation*}
	A(x) = \omega \times x + \eta, \quad x \in \Omega,
\end{equation*}
where $\eta \in \R^3$ is the base velocity, and the axis of rotation and angular velocity are, respectively, defined by the direction and magnitude of $\omega \in \R^3$. In particular, the space of all rigid motions can be defined as
\begin{equation} \label{eq_def_R}
	\mathcal{R} = \{\, A\colon \Omega \to \R^3  :   A(x) = B x + \eta \text{ for } \eta \in \R^3 \text{ and antisymmetric } B \in \R^{3 \times 3} \,\}
\end{equation}
since there is a one-to-one correspondence between linear maps of the form $x \mapsto \omega \times x$ and antisymmetric matrices.

Let
\begin{equation*}
	L^\infty_+(\Omega) = \{\, \varsigma \in L^\infty(\Omega) : \inf \varsigma > 0 \,\}.
\end{equation*}
We have
\begin{equation*}
	\inner{u,v}_{L^2(\Omega)^{m \times n}} = \int_\Omega u : v\, {\rm d} x, \quad u,v \in L^2(\Omega)^{m \times n},
\end{equation*}
where $n,m \in \N$. We reserve the notation
\begin{equation*}
	\inner{g, f} = \int_{\partial \Omega} g \cdot f \, {\rm d} S 
\end{equation*}
for the inner product of $L^2(\partial \Omega)^3$.

\section{Pure traction problem of linear elasticity} \label{sec_direct}

A deformation of an isotropic elastic body $\Omega \subset \R^3$ can  be described by the displacement field $u \colon \Omega \to \R^3$ that indicates how points in $\Omega$ are translated. We assume that $\Omega$ has a $C^{1,1}$-boundary and consider the pure traction problem (see,~e.g.,~\cite{Ci88})
\begin{align*}
	\nabla \cdot (\C\widehat{\nabla} u )  &= 0  \text{ in } \Omega \\
	\gamma_{\C} u &= g \text{ on }\partial\Omega,
\end{align*}
where $\widehat{\nabla} u  = \frac{1}{2}(\nabla u + \nabla u^\top)$ is the symmetric derivative or the strain tensor, $g$ is the applied boundary traction, $\C$ is the fourth order stiffness/elastic tensor, 
and\footnote{
	Note that this definition of $\gamma_{\mathbf C}$ is only valid for regular $u$, because of the trace operator. We can however extend the definition to the weak solutions given in \eqref{eq_weak}. For a weak solution $u \in H^1(\Omega)^3$, we define $\gamma_{\mathbf{C}} u \in L^2(\p \Omega)^3$ as the element in $L^2(\p\Omega)^3$ for which
	\begin{equation*} 
		\inner{\gamma_{\mathbf{C}} u, \, v|_{\p \Omega}} = B(u,v), \quad \forall v \in H^1(\Omega)^3. 
	\end{equation*}
}
\begin{equation} \label{eq_gamma}
	\gamma_{\mathbf C} u = \bigl((\C\widehat{\nabla} u ) \nu\bigr)|_{\p \Omega},
\end{equation}
with $\nu$ denoting the exterior unit normal of $\partial \Omega$.
As we consider isotropic materials, $\C$ can be given as 
\begin{equation*} 
	(\C A)_{ij} =  \lambda \Tr(A) \delta_{ij} + 2\mu A_{ij}, \quad A \in \R^{3 \times 3},
\end{equation*}
where $\delta_{ij}$ is Kronecker's delta and $\lambda,\mu \in L^\infty_+(\Omega)$ are the Lamé parameters. In particular, the differential operator reduces to the form 
\begin{equation*}
	\nabla \cdot (\C\widehat{\nabla} u ) = \nabla \cdot \big(  \lambda (\nabla \cdot u) I + 2 \mu \widehat{\nabla} u  \big).
\end{equation*}

Let $g \in L^2(\partial\Omega)^3$. We say that $u \in H^1(\Omega)^3$ is a weak solution if
\begin{equation} \label{eq_weak}
	B_{\lambda, \mu}(u,v) = \ell(v) 
\end{equation}
for all $v \in H^1(\Omega)^3$, where
\begin{align*}
	B_{\lambda, \mu}(u,v) &= \int_\Omega \lambda (\nabla \cdot u)(\nabla \cdot v) + 2 \mu \widehat{\nabla} u :\widehat{\nabla} v  \,\di x, \\[1mm]
	\ell(v) &= \int_{ \partial\Omega} g \cdot v \,{\rm d}S.
\end{align*}
As the space of rigid motions $\mathcal{R}$ from \eqref{eq_def_R} is a subspace of the nullspaces of both the symmetric derivative and the divergence,   this variational equation (i)~does not have a solution for all $g \in L^2(\partial\Omega)^3$ and, if a solution exists, (ii)~it is not unique. Indeed, as $B_{\lambda,\mu}(u,v)$ vanishes for all $v \in \mathcal{R}$, the weak problem cannot have a solution unless $g$ belongs to
\begin{equation*}
	\mathcal{R}^{\perp} = \{\, g\in L^2(\partial\Omega)^3 : \inner{g,r} = 0  \text{ for all } r\in \mathcal{R} \,\}.
\end{equation*}
On the other hand, if $u \in H^1(\Omega)^3$ is a weak solution, so is $u + w$ for any $w \in \mathcal{R}$. It turns out that these are the only obstructions to the unique solvability; see,~e.g.,~\cite{Ciarlet78,OSY92} for more general analysis on the solvability of the pure traction problem.

As in~\cite{Kaleem26} we search for weak solutions with $g \in \mathcal{R}^{\perp}$ from $H^1(\Omega)^3 / \mathcal R$. The weak problem is well-defined for $u \in H^1(\Omega)^3 / \mathcal R$, and one is allowed to consider the equivalence classes of $H^1(\Omega)^3 / \mathcal R$ as the test functions, since neither $B_{\lambda,\mu}$ nor $\ell$ is affected by additive changes of their arguments in the direction of $\mathcal{R}$. As will be apparent soon, we may extract a unique representative in the space of \emph{centered displacements}
\begin{equation} \label{eq_def_H}
	\mathcal{H}(\Omega) = \{\, u \in H^1(\Omega)^3  : \inner{u,r}_{L^2(\Omega)^3} = 0 \text{ for all } r\in \mathcal{R} \,\}.
\end{equation}
The second Korn inequality applies to $\mathcal H(\Omega)$, stating that
\begin{align}  \label{eq_korn2_2}
	\norm{u}_{H^1(\Omega)^3} \leq C_{\rm K} \norm{\widehat{\nabla}u}_{L^2(\Omega)^3}, \quad \forall u \in \mathcal{H}(\Omega),
\end{align}
with some constant $C_{\rm K} = C_{\rm K}(\Omega) > 0$. The second Korn inequality holds in fact for any  closed subspace $W \subset H^1(\Omega)^3$ for which  
\begin{equation*}
	W\cap\mathcal{R} = \{0\},
\end{equation*}
and hence for $\mathcal H(\Omega)$, since $\mathcal H(\Omega) \cap \mathcal R = \{ 0 \}$. For a proof of \eqref{eq_korn2_2}, see \cite[theorem 2.5]{OSY92}.

In consequence, $B_{\lambda, \mu}(u,v)$ is both coercive and continuous on $\mathcal{H}(\Omega)$:
\begin{align}  
	\abs{B_{\lambda, \mu}(u,u)} &\geq \frac{2 \inf \mu}{C_{\rm K}^2} \norm{u}_{H^1(\Omega)^3}^2, \label{eq:coer} \\
	\abs{B_{\lambda, \mu}(u,v)} &\leq ( 2 \sup\mu + 3 \sup \lambda ) \norm{u}_{H^1(\Omega)^3} \norm{v}_{H^1(\Omega)^3} \label{eq:cont}
\end{align}
for all $u, v \in \mathcal{H}(\Omega)$.

\begin{prop} \label{prop_wellposedness}
	There is a unique weak solution $[u] \in  H^1(\Omega)^3 / \mathcal{R}$ if and only if $g\in\mathcal{R}^{\perp}$. For $g\in\mathcal{R}^{\perp}$, the solution equivalence class $[u]$ has a unique representative $u\in\mathcal H(\Omega)$, and this $u$ minimizes the $L^2(\Omega)^3$-norm over elements in~$[u]$. Moreover, 
	\begin{equation}  \label{eq_aprioriEst}
		\norm{u}_{H^1(\Omega)^3} \leq C\norm{g}_{L^2(\partial\Omega)^3},
	\end{equation}
	where $C = C(\mu, \Omega) > 0$ is independent of $g$.
\end{prop}

\begin{proof} 
	The insolvability for $g \in L^2(\partial \Omega)^3 \setminus \mathcal{R}^\perp$ was already settled above. On the other hand, the unique solvability in $H^1(\Omega)^3 / \mathcal{R}$ for $g \in \mathcal{R}^\perp$, with continuous dependence on the data, is proven in \cite{Ciarlet78}; see also~\cite[theorem~3.4]{Kaleem26}. The fact that the corresponding solution $[u]$ has a unique representative $u$ in the space of centered displacements $\mathcal{H}(\Omega)$, with the claimed $L^2(\Omega)^3$-norm minimization property, follows directly from \cite[lemmas~5.2 \& 5.3]{Kaleem26}.
	
	As $u$ minimizes the $L^2(\Omega)^3$-norm instead of the $H^1(\Omega)^3$-norm, \eqref{eq_aprioriEst} is not an immediate consequence of the continuous dependence of $[u] \in H^1(\Omega)^3 / \mathcal{R}$ on $g \in \mathcal{R}^\perp$. However, choosing $v = u$ in the weak formulation and applying \eqref{eq:coer} yields
	\begin{equation*}
		\norm{u}_{H^1(\Omega)^3}^2 \leq  C' B_{\lambda,\mu}(u,u) = C' \ell(u) \leq C \norm{u}_{H^1(\Omega)^3} \norm{g}_{L^2(\partial \Omega)^3}, 
	\end{equation*}
	where the last step follows from the Cauchy-Schwarz inequality and the trace theorem.
\end{proof}

When we refer to the weak solution for $g \in \mathcal{R}^\perp$ in what follows, we mean the unique weak solution $u\in\mathcal{H}(\Omega)$, the existence of which is guaranteed by Proposition~\ref{prop_wellposedness}. 

\begin{rem} \label{rem_bc}
	The reason for considering the pure traction problem and resorting to the space of centered displacements $\mathcal H(\Omega)$ is to avoid imposing an auxiliary homogeneous Dirichlet condition on some part of the boundary $\partial\Omega$, as is done in,~e.g.,~\cite{EGH25,EH21,EP23}. The advantage of such a Dirichlet condition would be that it fixes $\Omega$ in place, which removes rigid body motions from the system. On the other hand, our choice enables avoiding problematic boundary conditions when constructing divergence-free localized solutions in Section~\ref{sec:divergence_free}.
\end{rem}

\subsection{The ND map and its properties}

The ND map $\Lambda_{\lambda, \mu}\colon \mathcal{R}^{\perp} \to L^2(\partial\Omega)^3$ (traction\hyp to\hyp displacement map) is defined as 
\begin{equation} \label{eq_ND_map}
	\Lambda_{\lambda, \mu}g = u|_{\partial\Omega},
\end{equation}
where $u\in \mathcal{H}(\Omega)$ is the corresponding unique weak solution for boundary traction $g\in\mathcal{R}^\perp$. 

\begin{rem}
	The operator $\Lambda_{\lambda, \mu}$ is symmetric in the sense that
	\begin{equation} \label{eq:self_adjoint}
		\inner{\Lambda_{\lambda,\mu} f, g} = B_{\lambda, \mu}(u_f, u_g) = \inner{\Lambda_{\lambda,\mu} g, f},
	\end{equation}
	where $u_f, u_g \in \mathcal{H}(\Omega)$ are the weak solutions for boundary tractions $f, g \in \mathcal{R}^\perp$. However, unlike in,~e.g.,~\cite{EH21,EGH25}, the identity \eqref{eq:self_adjoint} does not on its own give a characterization for $\Lambda_{\lambda,\mu}$ itself, but rather for $P \Lambda_{\lambda,\mu}$, where $P$ is the $L^2(\partial\Omega)^3$-orthogonal projection onto $\mathcal{R}^\perp$. That is, \eqref{eq:self_adjoint} would characterize the ND map if the trace of the representative $u$ of the corresponding solution equivalence class of $H^1(\Omega)^3 / \mathcal{R}$ in~\eqref{eq_ND_map} were chosen to be
orthogonal to $\mathcal{R}$ in $L^2(\partial \Omega)^3$, instead of the representative itself being orthogonal to $\mathcal{R}$ in 
$L^2(\Omega)^3$ (cf.~\eqref{eq_def_H}).
\end{rem}

Our subsequent analysis relies heavily on the so-called monotonicity inequalities that characterize the monotone dependence of $\Lambda_{\lambda,\mu}$ on $(\lambda, \mu) \in L^\infty_+(\Omega)^2$ in the Loewner order, as well as on properties of the Fr\'echet derivative of the forward map $\Lambda$. These are directly inherited from~\cite{EH21}, where the traction boundary condition is accompanied by an auxiliary homogeneous Dirichlet condition.

\begin{lem} \label{lem_monotonicity_ineq1}
	Assume $\mu_j,\lambda_j \in L^\infty_+(\Omega)$ for $j=1,2$. Denote $\Lambda_j = \Lambda_{\lambda_j, \mu_j}$ and let $u_j \in \mathcal{H}(\Omega)$ be weak solutions for $\mu=\mu_j$, $\lambda= \lambda_j$, and $g\in\mathcal{R}^\perp$. Then  
	\begin{align*}  
		\inner{(\Lambda_2-\Lambda_1)g, g } &\leq \int_{\Omega} (\lambda_1-\lambda_2)\abs{\nabla\cdot u_2}^2 + 2(\mu_1-\mu_2)\abs{\widehat{\nabla}u_2}^2\, \di x, \\[1mm]  
		\inner{(\Lambda_2-\Lambda_1)g, g} &\geq \int_{\Omega} \frac{\lambda_2}{\lambda_1}(\lambda_1-\lambda_2)\abs{\nabla\cdot u_2}^2 + 2\frac{\mu_2}{\mu_1}(\mu_1-\mu_2)\abs{\widehat{\nabla}u_2}^2 \, \di x. 
	\end{align*}
\end{lem}

\begin{proof}
	The result follows from the same line of reasoning as \cite[lemmas~2.1 \& 2.2]{EH21}.
\end{proof}

\begin{lem} \label{lem_Frechet}
	For any $(\lambda, \mu) \in L^\infty_+(\Omega)^2$, there exists a Fr\'echet derivative $\DLambda_{\lambda,\mu} \colon  L^\infty(\Omega)^2 \to \mathscr{L}(\mathcal{R}^\perp, L^2(\partial \Omega)^3)$ satisfying
	\begin{equation} \label{bilinear_Frechet}
		\inner{\DLambda_{\lambda,\mu}[h_\lambda,h_\mu] f,g} = -\int_\Omega h_\lambda (\nabla \cdot u_f)(\nabla \cdot u_g) + 2 h_\mu \widehat{\nabla} u_f :\widehat{\nabla} u_g  \,\di x,  \quad f, g \in \mathcal{R}^\perp,
	\end{equation}
	where $u_f, u_g \in \mathcal{H}(\Omega)$ are the weak solutions for the boundary tractions $f, g \in \mathcal{R}^\perp$, respectively.
\end{lem}

\begin{proof}
	Let $w \in \mathcal{H}(\Omega)$ be the solution to
	\begin{equation} \label{eq:Frechet}
		B_{\lambda, \mu}(w, v) = -B_{h_\lambda, h_\mu}(u_f,v) \quad \text{for all } v \in \mathcal{H}(\Omega),
	\end{equation}
	whose unique solvability is guaranteed by combining \eqref{eq:coer}--\eqref{eq:cont} with the Lax--Milgram theorem. By exploiting the uniform coercivity of the bilinear form $B_{\widetilde{\lambda}, \widetilde{\mu}}$ for $(\widetilde{\lambda}, \widetilde{\mu})$ in a small enough open neighborhood of $(\lambda, \mu) \in L^\infty_+(\Omega)^2$ as well as its linear dependence on the Lam\'e parameters, it follows straightforwardly that the Fr\'echet derivative is given by 
	\begin{equation*}
		\DLambda_{\lambda,\mu}[h_\lambda,h_\mu] f = w|_{\partial \Omega};
	\end{equation*}
	see,~e.g.,~\cite{Garde22} for a comprehensive analysis in a related setting. Finally, choosing $v = u_g$ in \eqref{eq:Frechet} proves \eqref{bilinear_Frechet}.
\end{proof}

As for the ND map in \eqref{eq:self_adjoint}, the identity \eqref{bilinear_Frechet} does not give a characterization for $\DLambda_{\lambda,\mu}$, but for its projected version $P \DLambda_{\lambda,\mu}$.

\section{Main results} \label{sec:main}

We will now consider monotonicity-based reconstruction related to $\supp(\mu-\mu_0)$ for a known \emph{constant} ``background'' shear modulus $\mu_0$, based on $\Lambda_{\lambda,\mu}$ for unknown $\lambda,\mu\in
L^\infty_+(\Omega)$. Specifically, we will call
\begin{equation*}
	D = \supp(\mu-\mu_0)
\end{equation*} 
the inclusions (which may consist of several connected components). We will assume
\begin{equation*}
	D \Subset \Omega,
\end{equation*}
and will reconstruct the \emph{outer shape} $D^\bullet$.
\begin{defn}
	The \emph{outer shape} $D^\bullet$ of $D$ is the smallest closed set with connected complement such that~$D\subseteq D^\bullet$. 
\end{defn}
Let $\lambda_0$ be some positive constant that we may choose as we like. We assume to know positive scalars satisfying
\begin{alignat*}{2}
	\alpha_\lambda &\leq \inf(\lambda), &\qquad \alpha_\mu &\leq \min\{\inf(\mu),\mu_0\}, \\
	\beta_\lambda &\geq \sup(\lambda), &\qquad \beta_\mu &\geq \max\{\sup(\mu),\mu_0\}.
\end{alignat*}
There are two main approaches for the monotonicity method: the \emph{outer approach} which is the most general, and the \emph{inner approach} that requires much stronger assumptions but may be more straightforward to implement numerically. 

A notable difference, compared to previous results, is that here the signs for $\mu-\mu_0$ and $\lambda-\lambda_0$ do not have to align, i.e.,~$(\mu-\mu_0)(\lambda-\lambda_0)$ is not required to be non-negative throughout $\Omega$. This will be achieved by using certain localizing solutions that are divergence-free.

We state the geometric assumptions on the sets on which we localize separately, as we will also refer to them later, throughout Section~\ref{sec:divergence_free}.
\begin{assumption} \label{assump_D1D2}
	Let $D_1, D_2 \subset \Omega$ be open and assume that $\partial D_1$ is Lipschitz continuous, $\partial D_2$ is smooth, and $D_2$ is non-empty. Moreover, assume that $D_1 \cap D_2 = \emptyset$ and $\Omega \setminus \overline{(D_1 \cup D_2)}$ is connected.
\end{assumption}
\begin{thm} \label{thm_loc}
	Let $D_1$ and $D_2$ satisfy Assumption~\ref{assump_D1D2}. There exist $g_j\in\mathcal{R}^\perp$ and $u_j\in\mathcal{H}(\Omega)$ satisfying 
	\begin{equation*}
		\nabla \cdot \big(\lambda_0 (\nabla \cdot u_j) I + 2 \mu_0 \widehat{\nabla} u_j  \big) = 0 \text{ in } \Omega \quad\text{and}\quad \gamma_{\C} u_j = g_j \text{ on }\partial\Omega,
	\end{equation*}
	and 
	\begin{equation*}
		\nabla\cdot u_j = 0 \text{ in } \Omega,
	\end{equation*}
	localized such that
	\begin{equation*}
		\norm{\widehat{\nabla} u_j}_{L^2(D'_1)^{3 \times 3}} \to 0 \quad \text{and} \quad \norm{\widehat{\nabla} u_j}_{L^2(D_2)^{3 \times 3}} \to \infty
	\end{equation*}
	as $j \to \infty$  for any $D'_1 \Subset D_1$. 
\end{thm}
\begin{proof}
	The proof is given in Section~\ref{sec:divergence_free}, culminating in Section~\ref{sec:proofthmloc}.
\end{proof}

\subsection{Outer approach} 

Here we consider the outer approach of the monotonicity method, for determining $D^\bullet$, where we allow $\mu-\mu_0$ to take both positive and negative values (see, e.g.,~\cite[Theorem~5.1]{EGH25}).

Consider the following set of admissible test inclusions:
\begin{equation*}
	\mathcal{A} = \{\, C \Subset \Omega : C \text{ is the closure of an open set and has connected complement} \,\}.
\end{equation*}
We define the following test operators for a measurable set $C$:
\begin{align*}
	\DLambda_C^+ &= \DLambda_{\lambda_0,\mu_0} [(\beta_\lambda-\lambda_0)\chi_\Omega,(\beta_\mu-\mu_0)\chi_C], \\
	\DLambda_C^- &= \DLambda_{\lambda_0,\mu_0} \big[\tfrac{\lambda_0}{\alpha_\lambda}(\alpha_\lambda-\lambda_0)\chi_\Omega,\tfrac{\mu_0}{\alpha_\mu}(\alpha_\mu-\mu_0)\chi_C \big],
\end{align*}
where $\chi$ denotes the characteristic function of the indicated set. Note that we use a characteristic function on $\Omega$ in the $\lambda$-direction; this is needed for the ``easy'' direction of the proof, to take care of the support of $\lambda-\lambda_0$ which may extend outside $D^\bullet$. In the ``difficult'' direction of the proof, this contribution is obviously not seen by the divergence-free localized sequence of solutions.

We moreover require the following technical assumption. 
\begin{assumption} \label{assump:technical}
	For every $x\in\partial D^\bullet$ and every open neighborhood $W$ of $x$, assume there exists a relatively open connected set $V\subset D^\bullet\cap W$ that intersects $\partial D^\bullet$, satisfying either of two options:
	\begin{enumerate}[\rm(a)]
		\item $\mu\geq \mu_0$ in $V$, and there exists an open ball $B\Subset V$ such that $\inf_B(\mu -\mu_0)>0$.
		\item $\mu\leq \mu_0$ in $V$, and there exists an open ball $B\Subset V$ such that $\sup_B(\mu -\mu_0)<0$.
	\end{enumerate}
\end{assumption}
Note, in particular, that the definiteness condition in Assumption~\ref{assump:technical} is only required near $\partial D^\bullet$, and at a positive distance from $\partial D^\bullet$, $\mu$ can be a general $L^\infty_+$ function (and $\lambda$ can be a general $L^\infty_+$ function in the whole domain). Moreover, the assumption does not require a jump away from $\mu_0$, but also allows, e.g.,~a continuous deviation in the form of a strict local increase/decrease away from $\mu_0$ when entering $D^\bullet$.

\begin{thm} \label{thm:outer} 
	For a measurable  set $C\subseteq\Omega$, we have
	\begin{equation*}
		D\subseteq C \quad \text{implies} \quad \DLambda_C^- \geq \Lambda_{\lambda,\mu} - \Lambda_{\lambda_0,\mu_0} \geq \DLambda_C^+.
	\end{equation*}
	Under Assumption~\ref{assump:technical}, for $C\in\mathcal{A}$ we have
	\begin{equation*}
		\DLambda_C^- \geq \Lambda_{\lambda,\mu}-\Lambda_{\lambda_0,\mu_0} \geq \DLambda_C^+ \quad \text{implies} \quad D\subseteq C.
	\end{equation*}
\end{thm}
\begin{proof}
	The proof is given in Section~\ref{sec:proofthmouter}.
\end{proof}
\begin{rem}
	Under the given assumptions, Theorem~\ref{thm:outer} gives
	\begin{equation*}
		D^\bullet = \cap\,\{\, C\in\mathcal{A} : \DLambda_C^- \geq \Lambda_{\lambda,\mu}-\Lambda_{\lambda_0,\mu_0} \geq \DLambda_C^+ \,\}.
	\end{equation*}
	Additionally, as observed in the proof:
	\begin{itemize}
		\item[$\diamond$] If $\mu\geq\mu_0$ everywhere near $\partial D^\bullet$ (only arriving at case (a) of Assumption~\ref{assump:technical}), then we only need to consider the inequality $\Lambda_{\lambda,\mu}-\Lambda_{\lambda_0,\mu_0} \geq \DLambda_C^+$. 
		\item[$\diamond$] If $\mu\leq\mu_0$ everywhere near $\partial D^\bullet$ (only arriving at case (b) of Assumption~\ref{assump:technical}), then we only need to consider the inequality $\DLambda_C^- \geq
		\Lambda_{\lambda,\mu}-\Lambda_{\lambda_0,\mu_0}$. 
	\end{itemize}
\end{rem}

\subsection{Inner approach}

In this section we consider the inner approach of testing if a ball $B$ is contained in $D$, which is only applicable to perturbations that are \emph{either} positive or negative throughout the inclusions. We need the following much stronger assumptions:
\begin{assumption} \label{assump:inner}
	Assume that either of the following two conditions holds:
	\begin{enumerate}[\rm(a)]
		\item \emph{Positive inclusions}: $\mu-\mu_0\geq c$ in $D$ for some known $c>0$.
		\item \emph{Negative inclusions}: $\mu-\mu_0\leq -c$ in $D$ for some known $c>0$.
	\end{enumerate}
\end{assumption}
We define the following test operators (different from the previous ones), for a measurable set $B\subseteq\Omega$,
\begin{align*}
\DhLambda_B^+ &= \DLambda_{\lambda_0,\mu_0}\big[\tfrac{\lambda_0}{\alpha_\lambda}(\alpha_\lambda-\lambda_0)\chi_\Omega,c\tfrac{\mu_0}{\beta_\mu}\chi_B \big], \\[1mm]
\DhLambda_B^- &= \DLambda_{\lambda_0,\mu_0}[(\beta_\lambda-\lambda_0)\chi_\Omega,-c\chi_B].
\end{align*}
Again, the perturbation in the $\lambda$-direction in all of $\Omega$ is to ensure that the ``easy'' direction of the proof works for any unknown $\lambda$, and this contribution is not seen in the proof of the ``difficult'' direction when using the divergence-free sequence of localized solutions.

\begin{prop} \label{prop:posinc}
	Assume that $\mu$ has \emph{positive inclusions} in $D$, corresponding to case~(a) of Assumption~\ref{assump:inner}. For a measurable set $B \subseteq \Omega$, we have 
	\begin{equation*}
		B\subseteq D \quad \text{implies} \quad \Lambda_{\lambda,\mu}-\Lambda_{\lambda_0,\mu_0} \leq \DhLambda^+_B.
	\end{equation*}
	For open set $B \subseteq \Omega$, we have
	\begin{equation*}
		\Lambda_{\lambda,\mu} - \Lambda_{\lambda_0,\mu_0} \leq \DhLambda^+_B \quad \text{implies} \quad B\subset D^\bullet.
	\end{equation*}
\end{prop}
\begin{proof}
	The proof is given in Section~\ref{sec:proofneginc}.
\end{proof}
\begin{prop} \label{prop:neginc}
	Assume that $\mu$ has \emph{negative inclusions} in $D$, corresponding to case~(b) of Assumption~\ref{assump:inner}. For a measurable set $B \subseteq \Omega$, we have 
	\begin{equation*}
		B\subseteq D \quad \text{implies} \quad \DhLambda^-_B \leq \Lambda_{\lambda,\mu}-\Lambda_{\lambda_0,\mu_0}.
	\end{equation*}
	For open set $B \subseteq \Omega$, we have
	\begin{equation*}
		\DhLambda^-_B \leq \Lambda_{\lambda,\mu}-\Lambda_{\lambda_0,\mu_0} \quad \text{implies} \quad B\subset D^\bullet.
	\end{equation*}
\end{prop}
\begin{proof}
	The proof is given in Section~\ref{sec:proofneginc}.
\end{proof}
\begin{rem}
	For positive inclusions, Proposition~\ref{prop:posinc} gives
	\begin{equation*}
		D^\circ \subseteq \cup\,\{\, B\subseteq\Omega \text{ open ball} : \Lambda_{\lambda,\mu} - \Lambda_{\lambda_0,\mu_0} \leq \DhLambda^+_B \,\} \subset D^\bullet,
	\end{equation*}
	and for negative inclusions, Proposition~\ref{prop:neginc} gives
	\begin{equation*}
		D^\circ \subseteq \cup\,\{\, B\subseteq\Omega \text{ open ball} : \DhLambda^-_B \leq \Lambda_{\lambda,\mu}-\Lambda_{\lambda_0,\mu_0} \,\} \subset D^\bullet,
	\end{equation*}
	where $D^\circ$ is the interior of the set $D$.
\end{rem}

\section{Divergence-free fields and localization} \label{sec:divergence_free}

In this section, we consider the pure traction problem with constant Lam\'e parameters. Section~\ref{sec_div_free_fields} first constructs a family of divergence-free $\mathcal H(\Omega)$-solutions by resorting to harmonic vector fields. Section~\ref{sec_loc_for_div_free} then studies localizing the energy of these solutions to subsets of $\Omega$. Finally, this leads to a proof of Theorem~\ref{thm_loc} in Section~\ref{sec:proofthmloc}.

\subsection{Divergence-free fields} \label{sec_div_free_fields} 

Consider the PDE with constant Lam\'e parameters $\lambda_0$ and $\mu_0$, 
\begin{equation} \label{eq:const_coef}
	\nabla \cdot \big(\lambda_0 (\nabla \cdot u) I + 2 \mu_0 \widehat{\nabla} u  \big) = 0 \text{ in } \Omega,
\end{equation}
without yet paying attention to the traction boundary condition. Our goal is to construct such divergence-free solutions to \eqref{eq:const_coef} that they form an infinite-dimensional subspace of $\mathcal{H}(\Omega)$. That is, we are looking for $u$ that satisfy 
\begin{equation} \label{eq:div_free}
	\nabla \cdot u = 0 \text{ in } \Omega
\end{equation}
in addition to \eqref{eq:const_coef}.

Since the Lam\'e parameters are assumed to be constants, a direct calculation verifies that \eqref{eq:const_coef} reduces to
\begin{equation*}
	\mu_0 \Delta u + (\lambda_0 + \mu_0) \nabla \bigl(\nabla \cdot u\bigr) = 0  \text{ in } \Omega.
\end{equation*}
Since 
\begin{equation*}
	\Delta u  = \nabla \bigl(\nabla \cdot u\bigr) - \nabla \times \bigl(\nabla \times u\bigr), 
\end{equation*}
the conditions \eqref{eq:const_coef} and \eqref{eq:div_free} are equivalent to requiring that
\begin{equation} \label{eq_curlcurl}
	\nabla \times \bigl(\nabla \times u\bigr) = 0 \quad \text{and} \quad \nabla \cdot u = 0.
\end{equation}
Obviously, any curl- and divergence-free field satisfies \eqref{eq_curlcurl}, which gives the motivation to look for suitable divergence-free solutions to \eqref{eq:const_coef} as gradients of harmonic potentials. To this end, consider $f \in H^{3/2}(\partial \Omega)$ and note that the Dirichlet problem
\begin{equation} \label{eq_laplace}
	\begin{split}
		\Delta \varphi &= 0  \text{ in } \Omega \\
		\varphi &= f  \text{ on } \partial\Omega,
	\end{split}
\end{equation}
has a unique solution in $H^2(\Omega)$ since $\partial \Omega$ is assumed to be of class $C^{1,1}$~\cite{Grisvard85}.

\begin{lem} \label{lem_Laplace_to_elastostatic}
	Assume $\lambda_0$ and $\mu_0$ are positive constants and let $\varphi \in H^2(\Omega)$ be the solution to \eqref{eq_laplace} with $f \in H^{3/2}(\partial \Omega)$. Then, $u = \nabla \varphi \in H^1(\Omega)^3$ satisfies
	\begin{equation} \label{eq_elastostatic_classic}
		\nabla \cdot \big(\lambda_0 (\nabla \cdot u) I + 2 \mu_0 \widehat{\nabla} u  \big) = 0 \quad \text{and} \quad \nabla \cdot u =0
	\end{equation}
	in $\Omega$. Furthermore, $\gamma_{\mathbf C} u \in \mathcal{R}^\perp$.
\end{lem}

\begin{proof}
    As a harmonic potential, $u \in H^1(\Omega)^3$ solves \eqref{eq_curlcurl}, and thus it also satisfies~\eqref{eq_elastostatic_classic}. Furthermore, multiplying the first equation in \eqref{eq_elastostatic_classic} by an arbitrary $v \in H^1(\Omega)^3$ and integrating by parts reveals that
    \begin{align*}  
		\int_\Omega \lambda_0 (\nabla \cdot u) (\nabla \cdot v) + 2 \mu_0 \widehat{\nabla} u :\widehat{\nabla} v \, {\rm d}x = \int_{ \partial\Omega} \gamma_{\mathbf C} u \cdot v \,{\rm d}S 
	\end{align*}
	for all $v \in H^1(\Omega)^3$. Thus, according to Proposition~\ref{prop_wellposedness}, it must hold that $\gamma_{\mathbf C} u \in \mathcal{R}^\perp$.
\end{proof}

Lemma \ref{lem_Laplace_to_elastostatic} demonstrates that one can construct a plethora of divergence-free weak $H^1(\Omega)^3$-solutions for constant Lam\'e parameters. However, for our considerations it is convenient to choose the Dirichlet boundary value in \eqref{eq_laplace} so that these solutions lie in $\mathcal{H}(\Omega)$. To this end, define
\begin{equation*} 
	\nu \cdot \mathcal{R} =  \{\, (\nu \cdot r)|_{\partial\Omega}  : r \in \mathcal R \,\}
\end{equation*}
and 
\begin{equation} \label{eq_f_cond}
    (\nu \cdot \mathcal{R})^\perp =  \{ \, f \in L^2(\partial \Omega)  :  \inner{f, \nu \cdot r} = 0 \text{ for all } r \in \mathcal{R} \, \}.
\end{equation}
The following proposition summarizes our entire construction.

\begin{prop} \label{prop_Laplace_to_elastostatic}
	Assume $\lambda_0$ and $\mu_0$ are positive constants and let $\varphi \in H^2(\Omega)$ be the solution to \eqref{eq_laplace} for $f \in H^{3/2}(\partial \Omega) \cap (\nu \cdot \mathcal{R})^\perp$. Then $u = \nabla \varphi$ is a divergence-free weak solution in $\mathcal{H}(\Omega)$ to~\eqref{eq:const_coef}.
\end{prop}

\begin{proof}
	Lemma~\ref{lem_Laplace_to_elastostatic} (and its proof) guarantees that $u \in H^1(\Omega)^3$ is a divergence-free weak solution, which means that it is enough to prove that $u \in \mathcal{H}(\Omega)$. Let $r = Bx + \eta \in \mathcal R$ be arbitrary (i.e., $B \in \R^{3 \times 3}$ and $\eta \in \R^3$ are, respectively, an arbitrary antisymmetric matrix and an arbitrary vector) and employ Green's formula to write
	\begin{align*}
		\int_\Omega \nabla \varphi \cdot (Bx + \eta) \,{\rm d} x  & = -\int_\Omega \varphi \nabla \cdot (Bx + \eta) \,\di x + \int_{\partial\Omega} (\nu \cdot r) f \, {\rm d} S = \int_{\partial\Omega} (\nu \cdot r) f\,\di S = 0, 
	\end{align*}
	which proves the claim.
\end{proof}

It is worth making two remarks related to the arguments in this section:

\begin{rem}
	The space of admissible Dirichlet boundary values \eqref{eq_f_cond} for finding solutions $\varphi$ to \eqref{eq_laplace} such that $u = \nabla \varphi \in \mathcal{H}(\Omega)$ is just the orthogonal complement of a finite-dimensional subspace in $L^2(\partial \Omega)$. Hence, we can employ a similar argument as in~\cite{HPS19b} for constructing localized solutions. See Section~\ref{sec_loc_for_div_free} below for the details. 
\end{rem}

\begin{rem} \label{rem_dirichlet}
	Let us briefly discuss the motivation for using the space of centered displacements $\mathcal H(\Omega)$ for achieving well-posedness. If we instead resorted to a homogeneous Dirichlet condition
	\begin{equation} \label{eq:homog_dirichlet}
		u|_{\Gamma_{\textup{D}}} = 0 
	\end{equation}
	on a non-empty $\Gamma_{\textup{D}} \subset \partial\Omega$ as in,~e.g.,~\cite{EGH25,EH21,EP26}, then the potential $\varphi$ of Lemma~\ref{lem_Laplace_to_elastostatic} would have to satisfy the additional condition 
	\begin{equation*}
		\nabla\varphi|_{\Gamma_{\textup{D}}} = 0.
	\end{equation*}
	Because both the tangential and normal derivatives of $\varphi$ would then vanish on $\Gamma_{\textup{D}}$, the harmonic function $\varphi$ would solve a Cauchy problem in $\Omega$ with constant Dirichlet and vanishing Neumann boundary values on $\Gamma_{\textup{D}}$, and it would thus have to be constant in $\Omega$ due to the principle of unique continuation. This means that under the condition \eqref{eq:homog_dirichlet}, the above construction would only guarantee the existence of the trivial divergence-free solution $u \equiv 0$ to \eqref{eq:const_coef}.
\end{rem}

\subsection{Localization for divergence-free fields} \label{sec_loc_for_div_free}

Next, we construct divergence-free solutions to \eqref{eq:const_coef} that localize in a given subset of $\Omega$. To this end, assume that $\D \subseteq \Omega$ is compact in the relative topology and has a Lipschitz boundary. We employ a Runge-type argument, which is similar to that in \cite{HPS19b}. 

According to Proposition~\ref{prop_Laplace_to_elastostatic}, gradients of functions in the space 
\begin{equation*}
	\mathcal{S} = \big\{\, \varphi \in  H^2(\D) : \varphi \text{ satisfies  \eqref{eq_laplace} for some } f \in H^{3/2}(\partial \Omega) \cap (\nu \cdot \mathcal{R})^\perp  \,\big\}
\end{equation*}
are divergence-free fields in $\mathcal{H}(\Omega)$ and satisfy \eqref{eq:const_coef}. We are particularly interested in the restrictions of these functions to the subset $\D$ and thus define
\begin{equation*}
	\mathcal{S}|_{\D} = \{\, \varphi|_{\D} : \varphi \in \mathcal{S} \,\}.
\end{equation*}
Our preliminary aim is to characterize $\overline{\mathcal{S}|_{\D}}$, where the closure is taken in the topology of $L^2(\D)$. This will indicate the functions that can be approximated in $H^{-1}(\D)^3$ by divergence-free solutions to \eqref{eq:const_coef}.

We characterize the closure $\overline{\mathcal{S}|_{\D}}$ in terms of the subspace
\begin{equation*} 
	\mathcal F = \{\, F \in L^2(\D) : \partial_\nu (S F)|_{\partial\Omega} \in  \nu \cdot \mathcal{R} \,\},
\end{equation*}
where $S \colon F \mapsto \varphi$ is the solution operator for the problem
\begin{equation} \label{eq_loc_source}
	\begin{split}
		\Delta \varphi &= F  \text{ in } \Omega \\
		\varphi &= 0  \text{ on } \partial\Omega.
	\end{split}
\end{equation}
Here and in what follows, $L^2(\D)$ is interpreted as a subspace of $L^2(\Omega)$ via zero-continuation. 

Now let
\begin{equation*} 
	\mathcal{S}_{\D} = \{\, \phi \in H^2(\D) : \Delta \phi = 0 \text{ and }  \phi \perp \mathcal F \,\},
\end{equation*} 
where the orthogonality is in the sense of $L^2(\D)$.

\begin{lem} \label{lem_S_D}
	It holds that $\overline{\mathcal{S}|_{\D}} = \mathcal{S}_{\D}$.
\end{lem}

\begin{proof}
	We only prove that $\mathcal{S}_{\D} \subseteq \overline{\mathcal{S}|_{\D}}$, since the opposite direction is not needed in what follows and it could be proven by a similar argument. More precisely, we aim to show that $(\mathcal{S}|_{\D})^{\perp} \subseteq \mathcal{S}_{\D}^\perp$, which proves the assertion since then $\mathcal{S}_{\D} \subseteq \mathcal{S}_{\D}^{\perp \perp} \subseteq (\mathcal{S}|_{\D})^{\perp \perp} = \overline{ \mathcal{S}|_{\D}}$.
	
	Let $\psi \in (\mathcal{S}|_{\D})^\perp$ and define $\phi \in H^2(\Omega)$ as the solution to
	\begin{align*} 
		\Delta\phi &= \chi_{\D} \psi \text{ in } \Omega\\
		\phi &= 0 \text{ on } \partial\Omega,
	\end{align*}
	where $\chi_{\D}$ denotes the characteristic function of $\D$. Then, for all $\varphi \in \mathcal{S}|_{\D}$,
	\begin{align*}
		0 = \inner{\varphi,\psi}_{L^2(\D)} = \inner{\varphi,\chi_{\D} \psi}_{L^2(\Omega)} = \inner{\varphi,\Delta \phi}_{L^2(\Omega)} = \inner{\partial_\nu \phi, \varphi}_{L^2(\partial\Omega)}.
	\end{align*}
	Since $\phi = S\psi$ and  $\varphi|_{\partial\Omega}$ can be any element of $H^{3/2}(\partial \Omega) \cap (\nu \cdot \mathcal{R})^\perp$, it follows from the density of the embedding $H^{3/2}(\partial \Omega) \hookrightarrow L^2(\partial \Omega)$ that 
	\begin{equation*}
		\partial_\nu S(\psi)|_{\partial\Omega} \in (\nu \cdot \mathcal{R})^{\perp \perp} = \nu \cdot \mathcal{R}.
	\end{equation*}
	Thus, $\psi \in \mathcal F$, which shows that $\psi \perp \mathcal{S}_{\D}$.
\end{proof}

Our objective is to localize a sequence of divergence-free solutions of \eqref{eq:const_coef} so that they are asymptotically small in some set $D_1$ and large in another set $D_2$. 

\begin{lem} \label{lem_blowupCand}
	Let $D_1$ and $D_2$ satisfy Assumption~\ref{assump_D1D2} and define $\D = D_1 \cup D_2$. There exists $\phi \in \mathcal{S}_{\D}$ such that 
	\begin{equation*}
		\phi \equiv 0 \text{ in } D_1, \quad \norm{\phi}_{ L^2(D_2)} \neq 0, \quad\text{and}\quad \norm{\widehat{\nabla} \nabla \phi}_{L^2(D_2)^{3\times3}} \neq 0. 
	\end{equation*}
\end{lem}

\begin{proof}
	To set the stage, let $\Lambda$ be the Dirichlet-to-Neumann map for the Laplacian in $D_2$ and define
	\begin{align*}
		\mathcal{X} = \operatorname{span} \{\, (\partial_\nu  - \Lambda )(S F)|_{\partial D_2} : F \in \mathcal F \,\},
	\end{align*}
	which is a finite-dimensional subspace of $L^{2}(\partial D_2)$, as reasoned in what follows: Since $F \in L^2(\D) \subseteq L^2(\Omega)$, the solution $SF$ to \eqref{eq_loc_source} belongs to $H^2(\Omega)$, meaning that $(S F)|_{\partial D_2} \in H^{3/2}(\p D_2 )$, $\Lambda (S F)|_{\partial D_2} \in H^{1/2}(\partial D_2)$, and  $\partial_\nu (S F)|_{\partial D_2} \in H^{1/2}(\partial D_2)$ by the trace theorem~\cite{Adams75} and standard regularity theory for elliptic PDEs~\cite{Grisvard85}. Hence, $\mathcal{X} \subseteq H^{1/2}(\partial D_2) \subseteq L^2(\partial D_2)$. 
	
	As $\mathcal{X}$ can be given as an image of 
	\begin{equation*}
		\mathcal{Y} = \operatorname{span}  \{\, (S F)|_{\Omega \setminus \overline{\D}} : F \in \mathcal F \,\},
	\end{equation*}
	under a linear mapping, the finite-dimensionality of $\mathcal{X}$ follows from that of $\mathcal{Y}$. Since $\dim (\nu \cdot \mathcal{R}) < \infty$ and due to linearity, the finite-dimensionality of $\mathcal{Y}$ can in turn be established by showing that  $\partial_\nu (SF)|_{\partial \Omega} = 0$ for $F \in L^2(\D)$ only if $(SF)|_{\Omega \setminus \overline{\D}} = 0$. To this end, observe that $\partial_\nu (SF)|_{\partial \Omega} = 0$ implies $SF \in H^2(\Omega)$ satisfies the Laplace equation in $\Omega \setminus \overline{\D}$ with vanishing Cauchy data on $\partial \Omega$.  Hence, $(S F)|_{\Omega \setminus \overline{\D}} = 0$ by the principle of unique continuation.  
	 
	We define 
	\begin{equation*}
		\phi = 
		\begin{dcases}
		0 &\text{ in } D_1, \\
		\widetilde{\phi} &\text{ in } D_2,
		\end{dcases}
	\end{equation*}
	where $\widetilde{\phi}$ solves the boundary value problem
	\begin{align*}
		\Delta\widetilde{\phi} &= 0 \text{ in } D_2 \\
		\widetilde{\phi} &= \widetilde{f} \text{ on } \partial D_2,
	\end{align*}
	with a non-zero Dirichlet value $\widetilde{f} \in H^{3/2}(\partial D_2)$ that is required to satisfy
	\begin{equation} \label{eq:tilde_ortho}
		\widetilde f \perp \bigl(\mathcal{X} + (\nu \cdot \mathcal{R})|_{\partial D_2} + \Span\{1\} \bigr)
	\end{equation}
	in the sense of $L^2(\partial D_2)$. Since $\mathcal{X} + (\nu \cdot \mathcal R)|_{\partial D_2} + \Span\{1\}$ inherits finite-dimensionality from the subspaces composing it, such  a non-zero $\widetilde f$ exists, producing $\phi \neq 0$. To complete the proof, we still need to check that $\phi$ belongs to $\mathcal{S}_{\D}$ and $\widehat{\nabla} \nabla \phi \neq 0$.
	
	As $\Delta \phi = 0$ in $\D$ and $\phi \in H^2(\D)$ due to $\widetilde{f}$ lying in $H^{3/2}(\partial D_2)$, we only need to confirm that $\phi \perp \mathcal{F}$ in the sense of $L^2(\D)$ in order to prove  that  $\phi \in \mathcal{S}_{\D}$. For an arbitrary $F \in \mathcal F$, Green's formula and the self-adjointness of $\Lambda$ yield
	\begin{align*} 
		\inner{\phi, F}_{L^2(\D)} &= \inner{\widetilde{\phi}, F}_{L^2(D_2)} = \inner{\widetilde{\phi} , \Delta SF}_{L^2(D_2)}  \\
		&= \inner{\widetilde{f}, \partial_\nu SF}_{L^2(\partial D_2)}  -  \inner{\Lambda \widetilde{f}, SF}_{L^2(\partial D_2)} \\
		&= \inner{\widetilde{f}, (\partial_\nu - \Lambda)SF}_{L^2(\partial D_2)} = 0,  
	\end{align*}
	where the last step follows from \eqref{eq:tilde_ortho}. Hence, $\phi \in \mathcal{S}_{\D}$.
	
	Since $\widetilde{f} \in H^{3/2}(\partial D_2)$ is orthogonal to $(\nu \cdot \mathcal R)|_{\partial D_2}$, Proposition~\ref{prop_Laplace_to_elastostatic} reveals that $\nabla \phi$ is in~$D_2$ a divergence-free weak $\mathcal{H}(D_2)$-solution to \eqref{eq:const_coef}. In particular, this guarantees that 
	\begin{equation} \label{eq:Korn2}
		\norm{\nabla \phi}_{L^2(D_2)^3} \leq C \norm{\widehat{\nabla} \nabla \phi}_{L^2(D_2)^{3 \times 3}}
	\end{equation}
	by Korn's inequality \eqref{eq_korn2_2}. Finally, as $\widetilde{f}$ is assumed to be non-constant (cf.~\eqref{eq:tilde_ortho}), $\phi$ is also non-constant, and thus the left-hand side of $\eqref{eq:Korn2}$ is positive. This completes the proof.
\end{proof}

By Lemma \ref{lem_blowupCand}, there exists $\phi \in \mathcal{S}_{\D}$ that is localized in $D_2$, and by Lemma \ref{lem_S_D}, it can be approximated by elements in $\mathcal{S}$. By combining these two observations, we get the following localization result.

\begin{prop} \label{prop_localization1}
	Let $D_1$ and $D_2$ satisfy Assumption~\ref{assump_D1D2}. There exists a sequence $(\varphi_j)$ in $\mathcal{S}$ such that 
	\begin{equation*}
		\norm{\varphi_j}_{L^2(D_1)} \to 0 \quad \text{and} \quad \norm{\varphi_j}_{L^2(D_2)} \to \infty
	\end{equation*}
	as $j \to \infty$.
\end{prop}

\begin{proof}
	Set $\D = D_1 \cup D_2$. By Lemma~\ref{lem_blowupCand}, there is a function $\phi \in \mathcal{S}_{\D}$ that satisfies $\norm{\phi}_{L^2(D_1)} = 0$, $\norm{\phi}_{L^2(D_2)} \neq 0$, and $\norm{\widehat{\nabla}\nabla\phi}_{L^2(D_2)} \not= 0$.\footnote{The third property is not needed in this proof, but it is utilized in the proof of Theorem~\ref{thm_loc}.} By Lemma~\ref{lem_S_D}, this $\phi$ can be approximated by a sequence $(\phi_j)$ in $\mathcal{S}$ in the sense that 
	\begin{equation} \label{eq:phi_conv}
		\norm{\phi_j - \phi}_{L^2(\D) } \leq \frac{1}{j^2}.
	\end{equation}
	Defining $\varphi_j = j \phi_j \in\mathcal{S}$, it follows that
	\begin{equation*}
		\norm{\varphi_j - j \phi}_{L^2(\D) } \leq \frac{1}{j}.
	\end{equation*}
	Hence,
	\begin{equation*}
		\norm{\varphi_j}_{L^2(D_1)} \leq \frac{1}{j} \quad \text{and} \quad \norm{\varphi_j}_{L^2(D_2)} \geq j \, \norm{\phi}_{L^2(D_2)} - \frac{1}{j}, 
	\end{equation*}
	which proves the claim.
\end{proof}

The previous proposition shows that we can localize potentials in $\mathcal{S}$ with respect to the $L^2(\Omega)$-norm. However, we are ultimately interested in the localization of the strain of their gradients that, according to Proposition~\ref{prop_Laplace_to_elastostatic}, are divergence-free weak solutions to the constant coefficient equation~\eqref{eq:const_coef}. We are now ready to prove Theorem~\ref{thm_loc}.

\subsection{Proof of Theorem~\ref{thm_loc}} \label{sec:proofthmloc}

Let $\phi \in \mathcal{S}_{\D}$, $\phi_j \in \mathcal{S}$, and $\varphi_j = j \phi_j \in \mathcal{S}$ be as in the proof of Proposition \ref{prop_localization1}. By construction,
\begin{align*}
	\Delta (\phi - \phi_j) = 0 \text{ in } \D.
\end{align*}
Since $D'_1 \Subset D_1$, interior elliptic regularity thus yields
\begin{equation*}
	\norm{\widehat{\nabla}\nabla\phi -\widehat{\nabla}\nabla\phi_j}_{L^2(D'_1)^{3 \times 3}} \leq C \norm{\phi - \phi_j}_{H^2(D'_1)} \leq C' \norm{\phi-\phi_j}_{L^2(D_1)} \leq  \frac{C'}{j^2},
\end{equation*} 
where the last step follows from \eqref{eq:phi_conv}. Since $\phi|_{D_1} = 0$, we get
\begin{equation*}
	\norm{\widehat{\nabla} \nabla \varphi_j}_{L^2(D'_1)} = j \norm{\widehat{\nabla} \nabla \phi_j}_{L^2(D'_1)} \leq \frac{C'}{j},
\end{equation*}
which proves the first part of the claim.

The proof can be completed by showing that $\norm{\widehat{\nabla}\nabla\varphi_j}_{L^2(D'_2)} \to \infty$ as $j \to \infty$ for some subset $D'_2 \subseteq D_2$. We pick such a $D'_2 \Subset D_2$ with 
\begin{equation*}
	\norm{\widehat{\nabla}\nabla\phi}_{L^2(D'_2)} \neq 0,
\end{equation*}
which is possible since $\norm{\widehat{\nabla}\nabla\phi}_{L^2(D_2)} \neq 0$ by construction. As in the first part of the proof, interior elliptic regularity gives
\begin{equation*}
	\norm{\widehat{\nabla}\nabla\phi - \widehat{\nabla}\nabla\phi_j}_{L^2(D'_2)^{3 \times 3}} \leq C \norm{\phi - \phi_j}_{L^2(D_2)} \leq \frac{C''}{j^2}.
\end{equation*}
Hence, by the reverse triangle inequality,
\begin{align*}
	\norm{\widehat{\nabla}\nabla\varphi_j} _{L^2(D'_2)^{3 \times 3}} &\geq j \norm{\widehat{\nabla}\nabla\phi}_{L^2(D'_2)^{3 \times 3}} - j \norm{\widehat{\nabla}\nabla\phi - \widehat{\nabla}\nabla\phi_j}_{L^2(D'_2)^{3 \times 3}} \\ 
	&\geq   j \norm{\widehat{\nabla}\nabla\phi}_{L^2(D'_2)^{3 \times 3}} - \frac{C''}{j}.
\end{align*}
The result follows with $u_j = \nabla\varphi_j$ by Proposition~\ref{prop_Laplace_to_elastostatic}. \hfill $\Box$
	
\section{Proof of Theorem~\ref{thm:outer}} \label{sec:proofthmouter}

In the proof we abbreviate $\Lambda = \Lambda_{\lambda,\mu}$ and $\Lambda_0 = \Lambda_{\lambda_0,\mu_0}$.

\textbf{Proof of ``$\boldsymbol{D\subseteq C \Rightarrow \DLambda_C^- \geq \Lambda - \Lambda_0 \geq \DLambda_C^+}$'':}	Let $u_0$ be the solution corresponding to $\lambda_0$, $\mu_0$ and the boundary traction $g$. Since $D\subseteq C$, Lemmas~\ref{lem_monotonicity_ineq1} and~\ref{lem_Frechet} give
\begin{align*}
	\inner{(\Lambda - \Lambda_0 - \DLambda_C^+)g,g} &\geq \int_\Omega (\lambda_0 - \lambda)\abs{\nabla\cdot u_0}^2\,\di x + \int_D 2(\mu_0-\mu)\abs{\widehat{\nabla}u_0}^2\,\di x \\
	&\hphantom{\geq{}} + (\beta_\lambda-\lambda_0)\int_\Omega\abs{\nabla\cdot u_0}^2\,\di x +2(\beta_\mu-\mu_0)\int_C \abs{\widehat{\nabla}u_0}^2\,\di x \\
	&\geq 2(\beta_\mu-\mu_0)\int_{C\setminus D} \abs{\widehat{\nabla}u_0}^2\,\di x \geq 0.
\end{align*}

Next we consider the other inequality. Note that
\begin{equation*}
	\frac{\lambda_0}{\lambda}(\lambda - \lambda_0) \geq \lambda_0-\frac{\lambda_0^2}{\alpha_\lambda} = \frac{\lambda_0}{\alpha_\lambda}(\alpha_\lambda-\lambda_0) \text{ in } \Omega,
\end{equation*}
and likewise
\begin{equation*}
	\frac{\mu_0}{\mu}(\mu - \mu_0) \geq \frac{\mu_0}{\alpha_\mu}(\alpha_\mu-\mu_0) \text{ in } \Omega.
\end{equation*}
Since $D\subseteq C$, Lemmas~\ref{lem_monotonicity_ineq1} and~\ref{lem_Frechet} give
\begin{align*}
	\inner{(\Lambda_0 - \Lambda + \DLambda_C^-)g,g} &\geq \int_{\Omega} \frac{\lambda_0}{\lambda}(\lambda - \lambda_0)\abs{\nabla\cdot u_0}^2\,\di x + \int_D 2\frac{\mu_0}{\mu}(\mu - \mu_0)\abs{\widehat{\nabla}u_0}^2\,\di x \\
	&\hphantom{\geq{}} + \frac{\lambda_0}{\alpha_\lambda}(\lambda_0-\alpha_\lambda)\int_\Omega \abs{\nabla\cdot u_0}^2\,\di x + 2\frac{\mu_0}{\alpha_\mu}(\mu_0-\alpha_\mu)\int_C\abs{\widehat{\nabla}u_0}^2\,\di x \\
	&\geq 2\frac{\mu_0}{\alpha_\mu}(\mu_0 - \alpha_\mu)\int_{C\setminus D} \abs{\widehat{\nabla}u_0}^2\,\di x \geq 0.
\end{align*}

\textbf{Proof of ``$\boldsymbol{\DLambda_C^- \geq \Lambda - \Lambda_0 \geq \DLambda_C^+ \Rightarrow D\subseteq C}$'':} Assume now that $D\not\subseteq C$. Since $D^\bullet,C\in\mathcal{A}$ we can localize in $D$ and away from $C$. In particular, one of two options holds:
\begin{itemize}
\item[$\diamond$] Case A: Let $V$ and $B$ be as in part (a) of Assumption~\ref{assump:technical} and where $V\cap C = \emptyset$.
\item[$\diamond$] Case B: Let $V$ and $B$ be as in part (b) of Assumption~\ref{assump:technical} and where $V\cap C = \emptyset$.
\end{itemize}
We consider these cases separately. 

As we shall see below, the combination of the two cases proves that $D\not\subseteq C$ implies either $\DLambda_C^- \not\geq \Lambda-\Lambda_0$ or $\Lambda-\Lambda_0 \not\geq \DLambda_C^+$, which is the contrapositive formulation of the statement we are proving. In particular, case A contradicts the second operator inequality, while case B contradicts the first operator inequality. 

For both cases, there exists an open set $D_1$, with $D\setminus V \Subset D_1$ and $C\Subset D_1$, satisfying Assumption~\ref{assump_D1D2} together with $D_2 = B$. We pick $g_j$ and $u_j$ (corresponding to $\lambda_0$ and $\mu_0$) according to Theorem~\ref{thm_loc} such that $u_j$ is divergence-free and satisfies
\begin{equation*}
	\lim_{j\to\infty}\int_B \abs{\widehat{\nabla}u_j}^2\,\di x = \infty \quad \text{and}\quad \lim_{j\to\infty}\int_{D\setminus V} \abs{\widehat{\nabla}u_j}^2\,\di x = \lim_{j\to\infty}\int_{C} \abs{\widehat{\nabla}u_j}^2\,\di x = 0.
\end{equation*}

\textbf{Case A:} By Lemmas~\ref{lem_monotonicity_ineq1} and~\ref{lem_Frechet}, we get
\begin{align*}
	\inner{(\Lambda - \Lambda_0 - \DLambda_C^+)g_j,g_j} &\leq \int_{D} 2\frac{\mu_0}{\mu}(\mu_0-\mu)\abs{\widehat{\nabla}u_j}^2\,\di x + 2(\beta_\mu-\mu_0)\int_C \abs{\widehat{\nabla}u_j}^2\,\di x. 
\end{align*}
The integrals over $C$ and $D\setminus V$ tend to zero. Moreover, in $V$ we have $\mu\geq \mu_0$ and in $B\Subset V$ we have $\inf_B(\mu-\mu_0)>0$. Thus we conclude that
\begin{equation*}
	\lim_{j\to\infty}\inner{(\Lambda - \Lambda_0 - \DLambda_C^+)g_j,g_j} = -\infty,
\end{equation*}
which means that $\Lambda-\Lambda_0 \not\geq \DLambda_C^+$.

\textbf{Case B:} By Lemmas~\ref{lem_monotonicity_ineq1} and~\ref{lem_Frechet}, we get
\begin{align*}
	\inner{(\Lambda_0 - \Lambda + \DLambda_C^-)g_j,g_j} &\leq \int_{D} 2(\mu - \mu_0)\abs{\widehat{\nabla}u_j}^2\,\di x + 2\frac{\mu_0}{\alpha_\mu}(\mu_0-\alpha_\mu)\int_C \abs{\widehat{\nabla}u_j}^2\,\di x.
\end{align*}
The integrals over $C$ and $D\setminus V$ tend to zero. Moreover, in $V$ we have $\mu\leq \mu_0$ and in $B\Subset V$ we have $\sup_B(\mu-\mu_0)<0$. Thus we conclude that
\begin{equation*}
	\lim_{j\to\infty}\inner{(\Lambda_0 - \Lambda + \DLambda_C^-)g_j,g_j} = -\infty,
\end{equation*}
which means that $\DLambda_C^- \not\geq \Lambda-\Lambda_0$.

\section{Proofs of Propositions ~\ref{prop:posinc} and ~\ref{prop:neginc}} \label{sec:proofneginc}

In this section  we use the  abbreviations $\Lambda = \Lambda_{\lambda,\mu}$ and $\Lambda_0 = \Lambda_{\lambda_0,\mu_0}$.

\medskip
\textbf{Proof of Proposition~\ref{prop:posinc}:}
We first prove that $B \subseteq D \Rightarrow \Lambda - \Lambda_0 \leq \DhLambda^+_B$. Let $u_0$ be the solution corresponding to $\lambda_0$, $\mu_0$ and the boundary traction $g$. We make use of 
\begin{equation*}
	\frac{\lambda_0}{\lambda}(\lambda - \lambda_0) \geq \lambda_0-\frac{\lambda_0^2}{\alpha_\lambda} = \frac{\lambda_0}{\alpha_\lambda}(\alpha_\lambda-\lambda_0) \text{ in } \Omega
\end{equation*} 
and 
\begin{equation*}
	\frac{\mu_0}{\mu}(\mu - \mu_0) \geq c\frac{\mu_0}{\beta_\mu} \text{ in } D,
\end{equation*}
combined with $B\subseteq D$ and Lemmas~\ref{lem_monotonicity_ineq1} and~\ref{lem_Frechet}: 
\begin{align*}
	\inner{(\DhLambda_B^+-\Lambda+\Lambda_0)g,g} &\geq \int_{\Omega} \frac{\lambda_0}{\lambda}(\lambda - \lambda_0)\abs{\nabla\cdot u_0}^2\,\di x + \int_D 2\frac{\mu_0}{\mu}(\mu - \mu_0)\abs{\widehat{\nabla}u_0}^2\,\di x \\
	&\hphantom{\geq{}} + \frac{\lambda_0}{\alpha_\lambda}(\lambda_0-\alpha_\lambda)\int_\Omega \abs{\nabla\cdot u_0}^2\,\di x - 2c\frac{\mu_0}{\beta_\mu}\int_B\abs{\widehat{\nabla}u_0}^2\,\di x \\
	&\geq 2c\frac{\mu_0}{\beta_\mu}\int_{D\setminus B} \abs{\widehat{\nabla}u_0}^2\,\di x \geq 0.
\end{align*}

Next we prove that $\Lambda-\Lambda_0 \leq \DhLambda^+_B \Rightarrow B\subset D^\bullet$. Assume $B\not\subset D^\bullet$. Since $D^\bullet$ is closed and has connected complement, while $B$ is open and non-empty, we may localize inside $B_0 \Subset B \setminus D^\bullet$ and away from $D$. Hence, there exists an open set $D_1$, with $D\Subset D_1$, satisfying Assumption~\ref{assump_D1D2} together with $D_2 = B_0$. Using Theorem~\ref{thm_loc}, there are $g_j$ and $u_j$ (corresponding to $\lambda_0$, $\mu_0$) such that $u_j$ is divergence-free and satisfies 
\begin{equation*}
	\lim_{j\to\infty}\int_{B_0} \abs{\widehat{\nabla}u_j}^2\,\di x = \infty \quad \text{and}\quad \lim_{j\to\infty}\int_{D} \abs{\widehat{\nabla}u_j}^2\,\di x = 0.
\end{equation*}
Lemmas~\ref{lem_monotonicity_ineq1} and~\ref{lem_Frechet} give
\begin{align*}
	\inner{(\DhLambda_B^+-\Lambda+\Lambda_0)g_j,g_j} &\leq \int_{\Omega} 2(\mu - \mu_0)\abs{\widehat{\nabla}u_j}^2\,\di x - 2c\frac{\mu_0}{\beta_\mu}\int_B\abs{\widehat{\nabla}u_j}^2\,\di x \\
	&\leq 2(\beta_\mu - \mu_0)\int_{D}\abs{\widehat{\nabla}u_j}^2\,\di x - 2c\frac{\mu_0}{\beta_\mu}\int_{B_0}\abs{\widehat{\nabla}u_j}^2\,\di x.
\end{align*}
Thus we conclude that
\begin{equation*}
	\lim_{j\to\infty}\inner{(\DhLambda_B^+-\Lambda+\Lambda_0)g_j,g_j} = -\infty,
\end{equation*}
which means that $\Lambda-\Lambda_0 \not\leq \DhLambda_B^+$.
\hfill $\Box$

\bigskip
\textbf{Proof of Proposition ~\ref{prop:neginc}:}
We first prove that $B\subseteq D \Rightarrow \DhLambda^-_B \leq \Lambda - \Lambda_0$. Let $u_0$ be the solution corresponding to $\lambda_0$, $\mu_0$ and the boundary traction $g$. Since $\mu_0-\mu \geq c$ in $D$ and $B\subseteq D$, Lemmas~\ref{lem_monotonicity_ineq1} and~\ref{lem_Frechet} give 
\begin{align*}
	\inner{(\Lambda-\Lambda_0-\DhLambda_B^-)g,g} &\geq \int_{\Omega} (\lambda_0 - \lambda)\abs{\nabla\cdot u_0}^2\,\di x + \int_D 2(\mu_0 - \mu)\abs{\widehat{\nabla}u_0}^2\,\di x \\
	&\hphantom{\geq{}} + (\beta_\lambda-\lambda_0)\int_\Omega \abs{\nabla\cdot u_0}^2\,\di x - 2c\int_B\abs{\widehat{\nabla}u_0}^2\,\di x \\
	&\geq 2c\int_{D\setminus B} \abs{\widehat{\nabla}u_0}^2\,\di x \geq 0.
\end{align*}

Next we prove that $\DhLambda^-_B \leq \Lambda - \Lambda_0 \Rightarrow B\subset D^\bullet$. Assume $B\not\subset D^\bullet$. Since $D^\bullet$ is closed and has connected complement, while $B$ is open and non-empty, we may localize inside $B_0 \Subset B \setminus D^\bullet$ and away from $D$. Hence, there exists an open set $D_1$, with $D\Subset D_1$, satisfying Assumption~\ref{assump_D1D2} together with $D_2 = B_0$. Using Theorem~\ref{thm_loc}, there are $g_j$ and $u_j$ (corresponding to $\lambda_0$ and $\mu_0$) such that $u_j$ is divergence-free and satisfies
\begin{equation*}
	\lim_{j\to\infty}\int_{B_0} \abs{\widehat{\nabla}u_j}^2\,\di x = \infty \quad \text{and}\quad \lim_{j\to\infty}\int_{D} \abs{\widehat{\nabla}u_j}^2\,\di x = 0.
\end{equation*}
We make use of
\begin{equation*}
	\frac{\mu_0}{\mu}(\mu_0 - \mu) = \frac{\mu_0^2}{\mu}-\mu_0 \leq \frac{\mu_0^2}{\alpha_\mu}-\mu_0 = \frac{\mu_0}{\alpha_\mu}(\mu_0 - \alpha_\mu) \text{ in } \Omega.
\end{equation*} 
Lemmas~\ref{lem_monotonicity_ineq1} and~\ref{lem_Frechet} give
\begin{align*}
	\inner{(\Lambda-\Lambda_0-\DhLambda_B^-)g_j,g_j} &\leq \int_{\Omega} 2\frac{\mu_0}{\mu}(\mu_0 - \mu)\abs{\widehat{\nabla}u_j}^2\,\di x - 2c\int_B\abs{\widehat{\nabla}u_j}^2\,\di x \\
	&\leq 2\frac{\mu_0}{\alpha_\mu}(\mu_0 - \alpha_\mu)\int_{D}\abs{\widehat{\nabla}u_j}^2\,\di x - 2c\int_{B_0}\abs{\widehat{\nabla}u_j}^2\,\di x.
\end{align*}
Thus we conclude that
\begin{equation*}
	\lim_{j\to\infty}\inner{(\Lambda-\Lambda_0-\DhLambda_B^-)g_j,g_j} = -\infty,
\end{equation*}
which means that $\DhLambda_B^- \not\leq \Lambda-\Lambda_0$.
\hfill $\Box$

\subsection*{Acknowledgements}

HG is supported by grant 10.46540/3120-00003B from Independent Research Fund Denmark. NH is supported by the Research Council of Finland (Flagship of Advanced Mathematics for Sensing, Imaging and Modelling grant 359181).
VP was supported by the Research Council of Finland (Flagship of Advanced Mathematics for Sensing, Imaging and Modelling grant 359186)
and by the Emil Aaltonen Foundation.


\begin{thebibliography}{10}

\bibitem{Adams75}
R. A. Adams. 
\newblock {\em Sobolev Spaces}. 
\newblock Academic Press, New York--London, 1975.

\bibitem{Barbone04}
P.~E. Barbone and  N.~H. Gokhale.
\newblock Elastic modulus imaging: on the uniqueness and nonuniqueness of the elastography inverse problem in two dimensions.
\newblock {\em Inverse Problems}, 20:283--296, 2004.

\bibitem{Beretta14a}
E. Beretta, E. Francini, A. Morassi, E. Rosset, and S. Vessella.
\newblock Lipschitz continuous dependence of piecewise constant {L}am\'e coefficients from boundary data: the case of non-flat interfaces.
\newblock {\em Inverse Problems}, 30:125005, 2014.

\bibitem{Beretta14b}
E. Beretta, E. Francini, and S. Vessella.
\newblock Uniqueness and Lipschitz stability for the identification of {L}am\'e parameters from boundary measurements.
\newblock {\em Inverse Probl. Imag.}, 8:611--644, 2014.

\bibitem{Garde2020}
V.~Candiani, J.~Dard\'e, H.~Garde, and N.~Hyv{\"o}nen.
\newblock Monotonicity-based reconstruction of extreme inclusions in electrical impedance tomography.
\newblock {\em SIAM J. Math. Anal.}, 52(6):6234--6259, 2020.

\bibitem{Carstea18}
C.~I. C\^{a}rstea, N. Honda, and G. Nakamura.
\newblock Uniqueness in the inverse boundary value problem for piecewise homogeneous anisotropic elasticity.
\newblock {\em SIAM J. Math. Anal.}, 50:3291--3302, 2018.
  
\bibitem{Ci88}
P. Ciarlet.
\newblock {\em Mathematical Elasticity, Volume 1: Three-dimensional Elasticity}.
\newblock North-Holland, Amsterdam, 1988.

\bibitem{Ciarlet78}
P. Ciarlet,
\newblock {\em The Finite Element Method for Elliptic Problems},
\newblock North-Holland, Amsterdam, 1978.

\bibitem{EGH25}
S. Eberle-Blick, H. Garde, and N. Hyv\"onen.
\newblock Direct reconstruction of general elastic inclusions.
\newblock {\em SIAM J. Math. Anal.}, accepted for publication.

\bibitem{EH21}
S. Eberle and B. Harrach. 
\newblock Shape reconstruction in linear elasticity: Standard and linearized monotonicity method.
\newblock {\em Inverse Problems}, 37:045006, 2021.

\bibitem{EP23}
S. Eberle--Blick and V. Pohjola.
\newblock The monotonicity method for inclusion detection and the time harmonic elastic wave equation.
\newblock \textit{Inverse Problems}, 40:045018, 2024.

\bibitem{EP26}
S. Eberle--Blick and V. Pohjola.
\newblock The linearized monotonicity method for elastic waves and the separation of material parameters.
\newblock \textit{SIAM J. Math. Anal.}, 58:2031-2137, 2026

\bibitem{Eskin02}
G. Eskin and J. Ralston.
\newblock On the inverse boundary value problem for linear isotropic elasticity.
\newblock {\em Inverse Problems}, 18:907--921, 2002.

\bibitem{Garde22}
H. Garde and N. Hyv\"onen. 
\newblock Series reversion in Calder\'on's problem. 
\newblock {\em Math. Comp.}, 91:1925-1953, 2022.

\bibitem{Harrach2008}
B. Gebauer.
\newblock Localized potentials in electrical impedance tomography.
\newblock {\em Inverse Probl. Imag.}, 2:251--269, 2008.

\bibitem{Grisvard85}
P. Grisvard, 
\newblock {\em Elliptic Problems in Nonsmooth Domains}, 
\newblock Pitman Advanced Publishing Program, Boston, 1985.

\bibitem{HPS19a} 
B. Harrach, V. Pohjola, M. Salo,
\newblock Dimension bounds in monotonicity methods for the Helmholtz equation,
\newblock \emph{SIAM J. Math. Anal.}, 51:2995-3019, 2019. 

\bibitem{HPS19b}
B. Harrach, V. Pohjola, and M. Salo.
\newblock Monotonicity and local uniqueness for the Helmholtz equation in a bounded domain.
\newblock \textit{Anal. PDE}, 12:1741--1771, 2019.  

\bibitem{HU13}
B. Harrach and M. Ullrich.
\newblock Monotonicity-based shape reconstruction in electrical impedance tomography,
\newblock \emph{SIAM J. Math. Anal.}, 45:3382--3403, 2013.

\bibitem{Ikehata90}
M. Ikehata.
\newblock Inversion formulas for the linearized problem for an inverse boundary value problem in elastic prospection.
\newblock {\em SIAM J. Appl. Math.}, 50:1635--1644, 1990.

\bibitem{Ikehata06}
M. Ikehata.
\newblock Stroh eigenvalues and identification of discontinuity in an anisotropic elastic material.
\newblock {\em Contemp. Math.}, 408:231--247, 2006.

\bibitem{Ikehata99}
M. Ikehata, G. Nakamura, and K. Tanuma.
\newblock Identification of the shape of the inclusion in the anisotropic elastic body.
\newblock {\em Appl. Anal.}, 72:17--26, 1999.

\bibitem{Kaleem26}
A. Kaleem, C. Gebhardt, and I. Romero.
\newblock  On the pure traction problem of linear elasticity:
A regularized formulation and its robust
approximation.
\newblock {\em Comput. Methods Appl. Mech. Eng.}, 459:119105, 2026.

\bibitem{Imanuvilov11}
O.~Y. Imanuvilov and M. Yamamoto.
\newblock On reconstruction of {L}am{\'e} coefficients from partial {C}auchy data.
\newblock {\em J. Inverse Ill-Posed Probl.}, 19:881--891, 2011.

\bibitem{Lin17}
Y.-H. Lin and G. Nakamura.
\newblock Boundary determination of the {L}am{\'e} moduli for the isotropic elasticity system.
\newblock {\em Inverse Problems}, 33:125004, 2017.

\bibitem{Nakamura99}
G. Nakamura, K. Tanuma, and G. Uhlmann.
\newblock Layer stripping for a transversely isotropic elastic medium.
\newblock {\em {SIAM} J. Appl. Math.}, 59:1879--1891, 1999.

\bibitem{Nakamura93}
G. Nakamura and G. Uhlmann.
\newblock Identification of {L}am{\'e} parameters by boundary measurements.
\newblock {\em Amer. J. Math.}, 115:1161--1187, 1993.

\bibitem{Nakamura94}
G. Nakamura and G. Uhlmann.
\newblock Global uniqueness for an inverse boundary value problem arising in elasticity.
\newblock {\em Invent. Math.}, 118:457--474, 1994.

\bibitem{Nakamura94erratum}
G. Nakamura and G. Uhlmann.
\newblock Erratum: ``Global uniqueness for an inverse boundary value problem arising in elasticity''.
\newblock {\em Invent. Math.}, 152:205--207, 2003.

\bibitem{Nakamura95}
G. Nakamura and G. Uhlmann.
\newblock Inverse problems at the boundary for an elastic medium.
\newblock {\em SIAM J. Math. Anal.}, 26:263--279, 1995.

\bibitem{OSY92}
O. Oleinik, S. Shamaev, and G. Yosifian.
\newblock {\em Mathematical Problems in Elasticity and Homogenization.}
\newblock North-Holland, Amsterdam, 1992.

\bibitem{Tamburrino02}
A. Tamburrino and G. Rubinacci.
\newblock A new non-iterative inversion method for electrical resistance tomography.
\newblock {\em Inverse Problems}, 18:1809--1829, 2002.

\end{thebibliography}
\end{document}